\documentclass[twocolumn]{autart}    

\usepackage{graphicx}          
\usepackage{xcolor}
\usepackage{amsmath, amsfonts}  
 \usepackage{algorithm}
\usepackage{algpseudocode}
 \usepackage{natbib}  
\usepackage{subcaption}

\newcommand*{\mb}[1]{\mathbf{#1}}%
\newcommand*{\mc}[1]{\mathcal{#1}}%
\newcommand*{\defeq}{\overset{\triangle}{=}}%
\newcommand*{\bx}{\mathbf{x}}%
\newcommand*{\R}{\mathbb{R}}%
\newenvironment{proof}{\textbf{Proof:} }{\hfill$\square$}

\usepackage{tikz,mdframed}
\usetikzlibrary{shapes.geometric, arrows}
\tikzstyle{startstop} = [rectangle, rounded corners, minimum width=3cm, minimum height=1cm,text centered, draw=black, fill=red!30]
\tikzstyle{process} = [rectangle, minimum width=3cm, minimum height=1cm, text centered, draw=black, fill=orange!30]
\tikzstyle{decision} = [diamond, minimum width=3cm, minimum height=1cm, text centered, draw=black, fill=green!30]
\tikzstyle{arrow} = [thick,->,>=stealth]
\newmdenv[
  linecolor=blue,
  linewidth=2pt,
  topline=false,
  bottomline=false,
  rightline=false,
  leftline=false,
  skipabove=\baselineskip,
  skipbelow=\baselineskip,
  innerleftmargin=10pt,
  innerrightmargin=10pt,
  innertopmargin=7pt,
  innerbottommargin=10pt,
  backgroundcolor=yellow!10
]{custombox}
\begin{document}

\begin{frontmatter}

\title{Anderson Acceleration for Distributed Constrained Optimization over Time-varying Networks\thanksref{footnoteinfo}} 

\thanks[footnoteinfo]{This paper was not presented at any IFAC meeting. Corresponding author Xuyang Wu. Email: wuxy6@sustech.edu.cn.}

\author[SUSTech]{Haijuan Liu}\ead{12431368@mail.sustech.edu.cn},    
\author[SUSTech]{Xuyang Wu}\ead{wuxy6@sustech.edu.cn}               
\address[SUSTech]{Southern University of Science and Technology, Shenzhen, China }    

\begin{abstract}                          
This paper applies the Anderson Acceleration (AA) technique to accelerate the Fenchel dual gradient method (FDGM) to solve constrained optimization problems over time-varying networks. AA is originally designed for accelerating fixed-point iterations, and its direct application to FDGM faces two challenges: 1) FDGM in time-varying networks cannot be formulated as a standard fixed-point update; 2) even if the network is fixed so that FDGM can be expressed as a fixed-point iteration, the direct application of AA is not distributively implementable. To overcome these challenges, we first rewrite each update of FDGM as inexactly solving several \emph{local} problems where each local problem involves two neighboring nodes only, and then incorporate AA to solve each local problem with higher accuracy, resulting in the Fenchel Dual Gradient Method with Anderson Acceleration (FDGM-AA). To guarantee global convergence of FDGM-AA, we equip it with a newly designed safe-guard scheme. Under mild conditions, our algorithm converges at a rate of \(O(1/\sqrt{k})\) for the primal sequence and \(O(1/k)\) for the dual sequence. The competitive performance of our algorithm is validated through numerical experiments.
\end{abstract}

\end{frontmatter}

\section{Introduction}
Many problems in networked systems are equivalent to finding the optimal solution of an optimization problem, such as economic dispatch in power systems \cite{yang2016distributed} and distributed learning \cite{lian2017can}. Distributed methods solve the problem through the collaboration of all nodes and via communications among neighboring nodes, which, compared to centralized methods, often yield better scalability and stronger privacy protection \cite{assran2020advances}.

In the last decades, a large number of distributed optimization methods have been proposed \cite{Aybat2018, margellos2017distributed, Nedic2017, Nedic2009, Nedic2010,qu2017harnessing, Shi2015, Shi2014, Wu2019, wu2022unifying, yang2016distributed,yi2021linear, Yuan2016}, most of which focus on consensus optimization \cite{Aybat2018, margellos2017distributed, Nedic2017, Nedic2009, Nedic2010, Shi2015, Shi2014, Wu2019, Yuan2016}. For consensus optimization, earlier methods are mainly primal and consensus-based, such as the distributed subgradient method \cite{Nedic2009}, the distributed projected subgradient method \cite{Nedic2010}, and the decentralized gradient descent method (DGD) \cite{Yuan2016}. These methods are simple but often yield slow convergence due to inexact convergence when using non-diminishing step-sizes. Some recent works mitigate this issue by proposing methods that can converge with fixed step-sizes, where two typical examples are EXTRA \cite{Shi2015} and DIGing \cite{Nedic2017,qu2017harnessing}. Specifically, EXTRA adds a correction term to DGD and DIGing combines DGD with a gradient-tracking technique, while they are lately shown to have a primal-dual nature \cite{wu2022unifying}. Other works in this category include the distributed ADMM method \citep{Aybat2018,Shi2014}, the Fenchel dual gradient method (FDGM)\cite{Wu2019}, and the distributed primal-dual method \cite{yi2021linear}, etc. Despite the rich literature, only a few can handle time-varying networks \cite{Nedic2009,Nedic2010,margellos2017distributed,Nedic2017,Wu2019}, among which FDGM \cite{Wu2019} is the only one that can simultaneously deal with local constraints and converge to the exact optimum with non-diminishing step-sizes.

To accelerate the solving of distributed constrained optimization over time-varying networks, we apply an Anderson acceleration (AA) technique \citep{Anderson1965} to the Fenchel dual gradient method. AA is an extrapolation technique for accelerating fixed-point iterations, and has been shown to be particularly effective in the centralized setting \cite{ZhangJz,mai2020anderson}. When applying AA to accelerate FDGM, it faces two challenges. First, FDGM cannot be formulated as a fixed-point iteration when the networks are time-varying. Second, even if the network is fixed so that FDGM can be described by the fixed-point iteration, the direct application of AA to FDGM cannot be implemented in a distributed way. To address these challenges, we first decouple FDGM as solving a set of local optimization problems, where each problem only involves two neighboring nodes. Then, we incorporate AA to solve each local problem, which gives FDGM-AA. We further design a safe-guard scheme to yield the global convergence of FGDM-AA. 

The main contributions of this paper are two-fold. First, we apply the effective Anderson acceleration technique to distributed optimization. Notably, the proposed FDGM-AA can solve optimization problems with local constraints over time-varying networks. Second, we theoretically prove that FDGM-AA can converge with fixed step-sizes under mild assumptions. The practical effectiveness of FDGM-AA is demonstrated by numerical experiments on classification tasks.

The rest of this paper is organized as follows: Section \ref{sec:Pre} gives problem formulation and preliminaries, and Section \ref{sec:alg_dev} develops the algorithm. Section \ref{sec:conv_ana} presents the convergence analysis. Section \ref{sec:Num_ex} presents numerical experimental results, and Section \ref{sec:Conc} concludes this paper.

\textbf{Notations and definitions:} We use \(\mathbb{R}^n\) to denote the \(n\)-dimensional Euclidean space and \(\|\cdot\|\) the Euclidean norm. For any two vectors $a,b\in\mathbb{R}^n$, \(\langle a, b \rangle\) represents their inner product, and for any two matrices $A$ and $B$, \(A\otimes B\) denotes their Kronecker product. For any $a\in \R$, \(\lfloor a\rfloor\) is the largest integer less than or equal to \(a\). We use $\mb{0}_n\in\R^n$, \(\mathbf{1}_n \in \mathbb{R}^n\), and $I_n\in \R^{n\times n}$ to denote the $n$-dimensional all-zero vector, all-one vector, and identity matrix, respectively, and ignore the subscript when it is clear from the context.

For any function $f: \mathbb{R}^n \to \mathbb{R}$, we say it is \emph{$\mu$-strongly convex} for some $\mu>0$ if $f(x)-\mu\|x\|^2/2$ is convex, and it is \emph{$L$-smooth} for some $L>0$ if it is differentiable and
$\| \nabla f({x}) - \nabla f({y}) \| \leq L \| {x} - {y} \|,\quad \forall x,y\in\mathbb{R}^n.$

\section{Consensus Optimization and Preliminaries}\label{sec:Pre}

This section first introduces consensus optimization over time-varying networks, and then reviews Fenchel dual gradient method and Anderson acceleration that are the foundations of our algorithm development in Section \ref{sec:alg_dev}.

\subsection{Consensus optimization on time-varying networks}

Consider a set $\mathcal{V}=\{1,\ldots,n\}$ of nodes, where each node $i\in\mathcal{V}$ has a local objective function \(f_i : \mathbb{R}^d \to \mathbb{R}\cup\{+\infty\}\). Consensus optimization aims at solving the minimum of the total objective function:
\begin{equation}\label{eq:problem}
\begin{array}{ll}
\underset{\bx\in\mathbb{R}^{nd}}{\operatorname{minimize}} ~& F(\mb{x})\defeq\sum_{i\in \mathcal{V}} f_i(x_i)\allowdisplaybreaks\\
\operatorname{subject~to} ~& x_1=\ldots=x_n,
\end{array}
\end{equation}
where $\mb{x}=(x_1^T, \ldots,x_n^T)^T$ and each $f_i$ is strongly convex.
\begin{assum}\label{assum:str_convex}
    Each \(f_i\) is $\mu$-strongly convex for some \(\mu>0\), and the interior of \(\bigcap_{i\in\mc{V}}~\mathrm{dom} f_i\) is non-empty.
\end{assum}
Note that Assumption \ref{assum:str_convex} requires neither the smoothness of $f_i$ nor its Lipschitz continuity, and allows each $f_i$ to incorporate the indicator function of a convex set.

We aim to solve problem \eqref{eq:problem} over networks with time-varying topologies. The networks are described by a sequence of undirected graphs \(\mathcal{G}^k = (\mathcal{V}, \mathcal{E}^k)\) for all \(k \ge 0\) where $\mathcal{V}$ is the vertex set and \(\mathcal{E}^k\) is the edge set at time \(k\). We also define $\mc{N}_i^k=\{j\mid \{i,j\}\in\mc{E}^k\}$ as the neighbor set of node $i$ at time $k$.


\subsection{Preliminaries}

\subsubsection{Fenchel dual gradient method}

The Fenchel dual gradient method (FDGM) \citep{Wu2019} solves problem \eqref{eq:problem} by solving its Fenchel dual problem:
\begin{equation}\label{eq:fenchel}
\begin{split}
     \underset{\mb{w}\in\R^{nd}}{\operatorname{minimize}} ~&~D(\mb{w})\defeq\sum_{i\in\mc{V}} d_i(w_i)\\
     \operatorname{subject~to} ~&~\sum_{i\in\mc{V}} w_i=\mb{0},
\end{split}
\end{equation} 
where $\mb{w}=(w_1^T, \ldots, w_n^T)^T$, \(d_i(w_i)=\max_{x_i} ~w_i^Tx_i - f_i(x_i)\) is the convex conjugate of \(f_i\), and \(\{w_i\}_{i\in\mathcal{V}}\) is the dual variable. Under Assumption \ref{assum:str_convex}, strong duality holds between problems (\ref{eq:problem}) and (\ref{eq:fenchel}), i.e., if \(\{w_i^{\star}\}_{i\in\mc{V}}\) is an optimal solution of the dual problem (\ref{eq:fenchel}), then \(\{\bar{x}_i(w_i^\star)\}_{i\in\mc{V}}\) is the optimal solution of the primal problem (\ref{eq:problem}), where, for any $w_i$, $\bar{x}_i(w_i) \defeq\arg\max_x~w_i^T x-f_i(x)$.
Moreover, Assumption \ref{assum:str_convex} guarantees the differentiability and $1/\mu$-smoothness of \(d_i\), where $\nabla d_i(w_i)=\bar{x}_i(w_i)$.

Given the differentiability of $D$, FDGM solves problem \eqref{eq:fenchel} by the weighted gradient methods:
\begin{equation}\label{eq:FDGM}
    \mb{w}^{k+1} = \mb{w}^k - \beta(H_{\mathcal{G}^k}\otimes I_d)\nabla D(\mb{w}^k),
\end{equation}
where $\beta>0$ is the step-size. For the matrix $H_{\mathcal{G}^k}\in\R^{n\times n}$, it is defined as: for a set of $\{h_{ij}^k\}_{\{i,j\}\in\mc{E}^k}$ satisfying $\underline{h}=\min_{k\ge 0}\min_{\{i,j\}\in\mc{E}^k} h_{ij}^k>0$, $[H_{\mathcal{G}^k}]_{ij}=-h_{ij}^k$ for $\{i,j\}\in\mc{E}^k$, $[H_{\mathcal{G}^k}]_{ii} =
\sum_{\ell\in \mathcal{N}_i^k} h_{i\ell}^k$, and $[H_{\mathcal{G}^k}]_{ij}=0$ otherwise.

By the structure of $H_{\mathcal{G}^k}$, we have 
\[\operatorname{Range}(H_{\mathcal{G}^k}\otimes I_d)\subseteq\{\mb{w}|~w_1+\ldots+w_n=\mb{0}\},\]
Therefore, if $\mb{w}^0$ is feasible to problem \eqref{eq:fenchel}, so are $\mb{w}^k$ $\forall k\ge 0$. Under mild conditions, FDGM converges to the optimal solution of \eqref{eq:fenchel} with fixed step-sizes, which, compared to existing methods that use diminishing step-sizes \citep{Nedic2010,margellos2017distributed}, usually yields faster practical convergence.

\begin{rem}
    For each node $i$ and at each iteration $k$, computing $\nabla d_i(w_i^k)$ and gathering $\nabla d_j(w_j^k)$, $j\in\mc{N}_i^k$ are the most expensive parts in terms of computation and communication, respectively. However, these dual gradients are used only once, causing the waste of information. By incorporating these historical information, the performance of FDGM may be further improved.
\end{rem}

\subsubsection{Anderson acceleration}
Anderson acceleration (AA) \citep{Anderson1965} uses historical information to accelerate fixed-point iterations:
\begin{equation}\label{eq:fixed-point}
    x^{k+1}=T(x^k),
\end{equation}
where \(T: \mathbb{R}^d \to \mathbb{R}^d\) is an operator. By assigning different $T$, the fixed-point iteration \eqref{eq:fixed-point} can describe a broad range of optimization methods. 

The basic idea of AA is to minimize an approximate optimality residual in a space formed by historical iterates. Specifically, at the $k$th iteration and for some $m\le k$, AA first considers the point of the form
\begin{equation}\label{eq:AA_xk1/2}
    x^{k+\frac{1}{2}}=\sum_{t=1}^m \alpha^{t,k}x^{k-m+t},
\end{equation}
where $\sum_{t=1}^m \alpha^{t,k}=1$. Define \(r^t = T(x^t) - x^t\), \(R^k=[r^{k-m+1}, \cdots, r^k]\), and $\alpha^k=(\alpha^{1,k}, \ldots, \alpha^{m,k})^T$. Next, AA approximates the residual at $x^{k+\frac{1}{2}}$ by
\begin{equation}\label{eq:AA_approx}
    T(x^{k+\frac{1}{2}})-x^{k+\frac{1}{2}}\approx R^k\alpha^k,
\end{equation}
and solve the ``optimal'' $\alpha^k$ by minimizing the residual
\begin{equation}\label{eq:uncon_alpha}
    \alpha^k=\arg\min_{\alpha\in\R^m: \mb{1}^T\alpha = 1} \left\|R^k \alpha \right\|_2.
\end{equation}
Finally, AA sets
\begin{equation}\label{eq:AA_final}
    x^{k+1} = \sum_{t=1}^m \alpha^{t,k} T(x^{k-m+t}),
\end{equation}
where the right-hand side approximates $T(x^{k+\frac{1}{2}})$. When $T$ is an affine operator, the approximation \eqref{eq:AA_approx} is accurate and $T(x^{k+\frac{1}{2}})=\sum_{t=1}^m \alpha^{t,k} T(x^{k-m+t})$. Although in the development of AA, $T$ is treated as an affine operator, in practice AA is widely used for non-affine operators and, to guarantee the global convergence of AA, additional safe-guard conditions are often incorporated.

\section{Algorithm Development}\label{sec:alg_dev}

In this section, we accelerate FDGM with AA, which faces two challenges. First, AA is designed for fixed-point iteration, while FDGM \eqref{eq:FDGM} is not a fixed-point iteration when the networks are time-varying. Second, even if the networks are fixed so that \eqref{eq:FDGM} can be described by \eqref{eq:fixed-point}, solving \eqref{eq:uncon_alpha} requires information from all nodes and does not have a distributed closed-form solution.

To address the two challenges, we equivalently transform the update \eqref{eq:FDGM} into approximate solving of a set of optimization problems, where each problem only involves two neighboring nodes. In particular, by restricting $\sum_{j \in \mathcal{N}_i^k} h_{ij}^k \le 1$, FDGM \eqref{eq:FDGM} can be equivalently rewritten as
\begin{align}
    w_i^{k+1} &= (1-\sum_{j\in\mathcal{N}^k_i} h^k_{ij})w_i^{k} +\sum_{j\in\mathcal{N}^k_i} h^k_{ij}w_{ij}^{k+\frac{1}{2}},\label{eq:separable_FDGM}
\end{align}
where for each $j\in\mc{N}_i^k$,
\begin{equation}\label{eq:wij1/2}
    w_{ij}^{k+\frac{1}{2}} = w_i^k-\beta(\nabla d_i(w_i^k)-\nabla d_j(w_j^k)).
\end{equation}
For each $\{i,j\}\in\mc{E}^k$, $(w_{ij}^{k+\frac{1}{2}}, w_{ji}^{k+\frac{1}{2}})$ is the result after one projected gradient descent step for solving
\begin{equation}\label{eq:gossip}
    \begin{array}{cc}
        \operatorname{minimize}~& d_i(w_{ij})+d_j(w_{ji})\\
        \operatorname{subject~to}~& w_{ij}+w_{ji} = w_i^k+w_j^k.
    \end{array}
\end{equation}
Following the above derivations, FDGM \eqref{eq:FDGM} can be viewed as \eqref{eq:separable_FDGM} where each pair of $(w_{ij}^{k+\frac{1}{2}}, w_{ji}^{k+\frac{1}{2}})$ is an approximate solution to \eqref{eq:gossip}.

\subsection{AA for solving \eqref{eq:gossip}}\label{ssec:gossip_AA}

We accelerate the update \eqref{eq:separable_FDGM} with AA to solve \eqref{eq:gossip}, which consists of four steps.

{\bf Step 1: dual gradient approximation}. We use $\nabla d_i,\nabla d_j$ at historical iterates to approximate their values at other points following the procedure of AA. Suppose that at the $k$th iteration, each node $i$ has a set $\mc{I}_{ij}^k$ of dual iterates and the corresponding dual gradients. The linear combination of these iterates is
\begin{equation}\label{eq:tilde_w}
    \tilde{w}_{ij}^{k+\frac{1}{2}}=\sum_{t\in \mc{I}_{ij}^k} \alpha_{ij}^{t,k} w_i^t
\end{equation}
where $\sum_{t\in \mc{I}_{ij}^k} \alpha_{ij}^{t,k}=1$, and we make the approximation
\begin{equation}\label{eq:AA_FDGM_approx}
    \nabla d_i(\tilde{w}_{ij}^{k+\frac{1}{2}})\approx \sum_{t\in \mc{I}_{ij}^k} \alpha_{ij}^{t,k} \nabla d_i(w_i^t).
\end{equation}
The above process is inspired by \eqref{eq:AA_xk1/2}--\eqref{eq:AA_approx} in AA. 

{\bf Step 2: coefficient determination.} We determine the coefficients $\{\alpha_{ij}^{t,k}\}_{t\in\mc{I}_{ij}^k}$ based on the KKT condition of \eqref{eq:gossip}: A pair of $(w_{ij}, w_{ji})$ is optimal to \eqref{eq:gossip} iff
\begin{equation}\label{eq:KKT}
    \nabla d_i(w_{ij}) = \nabla d_j(w_{ji}),\quad w_{ij}+w_{ji} = w_i^k+w_j^k.
\end{equation}
We approximate \eqref{eq:KKT} at $(\tilde{w}_{ij}^{k+\frac{1}{2}},\tilde{w}_{ji}^{k+\frac{1}{2}})$ using \eqref{eq:AA_FDGM_approx}:
\begin{equation}\label{eq:coefficients}
    \begin{array}{ll}
         \underset{\alpha_{ij}^k,\alpha_{ji}^k}{\operatorname{minimize}}~ &\left\|D_{ij}^k\alpha_{ij}^k-D_{ji}^k\alpha_{ji}^k\right\|^2\\
         \operatorname{subject~to}~ & W_{ij}^k\alpha_{ij}^k + W_{ji}^k\alpha_{ji}^k= w_i^k+w_j^k,\\
         &\mathbf{1}^T\alpha_{ij}^k=1,~ \mathbf{1}^T \alpha_{ji}^k=1,
    \end{array}
\end{equation}
where $D_{ij}^k$ and $W_{ij}^k$ are the stacking of $\nabla d_i(w_i^t)$ and $w_i^t$, $t\in\mc{I}_{ij}^k$, respectively. In \eqref{eq:coefficients}, $D_{ij}^k\alpha_{ij}^k$ and $D_{ji}^k\alpha_{ji}^k$ approximate $\nabla d_i(\tilde{w}_{ij}^{k+\frac{1}{2}})$ and $\nabla d_j(\tilde{w}_{ji}^{k+\frac{1}{2}})$, respectively, and we penalize their difference to approximate \eqref{eq:KKT}.

{\bf Step 3: projected steepest descent with approximate gradients.} We further perform a projected steepest descent based on the approximate gradient \eqref{eq:AA_FDGM_approx}:
\begin{equation}\label{eq:wij_bar}
    \begin{pmatrix}
    \bar{w}_{ij}^{k+\frac{1}{2}}\\
    \bar{w}_{ji}^{k+\frac{1}{2}}
\end{pmatrix}
=
\begin{pmatrix}
    \tilde{w}_{ij}^{k+\frac{1}{2}}\\
    \tilde{w}_{ji}^{k+\frac{1}{2}}
\end{pmatrix}
-\beta
\begin{pmatrix}
    D_{ij}^k\alpha_{ij}^k-D_{ji}^k\alpha_{ji}^k
 \\
   D_{ji}^k\alpha_{ji}^k-D_{ij}^k\alpha_{ij}^k
\end{pmatrix}.
\end{equation}
This step corresponds to \eqref{eq:AA_final} in AA.

{\bf Step 4: a safe-guard scheme.} Steps 1-3 are based on the approximation \eqref{eq:AA_FDGM_approx} while it is accurate only when $\nabla d_i$ is an affine operator. The inexact approximation may cause the algorithm to diverge even for the centralized AA, e.g., proximal gradient method with AA \citep{mai2020anderson}. To guarantee global convergence, we propose a safe-guard scheme, which accepts $(\bar{w}_{ij}^{k+\frac{1}{2}}, \bar{w}_{ji}^{k+\frac{1}{2}})$ only when it yields sufficient descent in the objective value: for some \(c_1, c_2 > 0\), set \((w_{ij}^{k+\frac{1}{2}}, w_{ji}^{k+\frac{1}{2}}) = (\bar{w}_{ij}^{k+\frac{1}{2}}, \bar{w}_{ji}^{k+\frac{1}{2}})\) if
\begin{equation}\label{eq:fixed_safe}
 \begin{array}{ll}
& d_i(\bar{w}_{ij}^{k+\frac{1}{2}}) + d_j(\bar{w}_{ji}^{k+\frac{1}{2}}) - d_i(w_i^k) - d_j(w_j^k) \\ 
& \le \min\big\{-c_1\big\| \nabla d_i(w_i^k) - \nabla d_j(w_j^k) \big\|^2,\\
&-c_2\big( \|\bar{w}_{ij}^{k+\frac{1}{2}} - w_i^k\|^2 + \|\bar{w}_{ji}^{k+\frac{1}{2}} - w_j^k\|^2\big) \big\}.
\end{array}
\end{equation}
Otherwise, set \((w_{ij}^{k+\frac{1}{2}}, w_{ji}^{k+\frac{1}{2}})\) according to \eqref{eq:wij1/2}. The verification of \eqref{eq:fixed_safe} requires node $i$ to know $d_j(\bar{w}_{ji}^{k+\frac{1}{2}})$, which leads to higher computation and communication costs but can be replaced by a simpler sufficient condition
\begin{equation}\label{eq:simple_safe}
\begin{array}{ll}
& \langle \nabla d_i(w_i^k), \bar{w}_{ij}^{k+\frac{1}{2}}-w_i^k\rangle+\frac{L}{2}\|\bar{w}_{ij}^{k+\frac{1}{2}}-w_i^k\|^2\\
&+\langle \nabla d_j(w_j^k), \bar{w}_{ji}^{k+\frac{1}{2}}-w_j^k\rangle+\frac{L}{2}\|\bar{w}_{ji}^{k+\frac{1}{2}}-w_j^k\|^2\\ 
& \le \min\big\{-c_1\big\| \nabla d_i(w_i^k) - \nabla d_j(w_j^k) \big\|^2,\\
&-c_2\big( \|\bar{w}_{ij}^{k+\frac{1}{2}} - w_i^k\|^2 + \|\bar{w}_j^{k+\frac{1}{2}} - w_j^k\|^2\big) \big\},
\end{array}
\end{equation}
where $L=1/\mu$ is the smoothness parameter of $d_i$ and the left-hand side is an upper bound of that of \eqref{eq:fixed_safe}. Moreover, \eqref{eq:fixed_safe} yields a favorable property, which is the key to our convergence analysis in Section \ref{sec:conv_ana}.

\begin{lem}\label{lem:always_descent}
    Suppose that Assumption \ref{assum:str_convex} holds. If
    \begin{equation}\label{eq:step-cond}
        \beta\in (0, 1/L),
    \end{equation}
    then with the safe-guard scheme, it holds that
    \begin{equation}\label{eq:always_descent}
    \begin{array}{ll}
        & d_i(w_{ij}^{k+\frac{1}{2}}) + d_j(w_{ji}^{k+\frac{1}{2}}) - d_i(w_i^k) - d_j(w_j^k)\\
        & \le \min\big\{-\theta_1\left\| \nabla d_i(w_i^k) - \nabla d_j(w_j^k) \right\|^2,\\
        &-\theta_2\big( \|w_{ij}^{k+\frac{1}{2}} - w_i^k\|^2 + \|w_{ji}^{k+\frac{1}{2}} - w_j^k\|^2\big) \big\},
    \end{array}
    \end{equation}
    where $\theta_1\!=\!\min\{\beta(1\!-\!\beta L), c_1\},\theta_2\!=\!\min\{(1/\beta-L)/2,c_2\}$.
    \end{lem}
    \begin{proof}
        See Appendix A.
    \end{proof}

\subsection{Main algorithm} 

By integrating the gossip update in Section \ref{ssec:gossip_AA} into the aggregation scheme \eqref{eq:separable_FDGM}, we propose a Fenchel Dual Gradient Method with Anderson Acceleration (FDGM-AA) algorithm. To introduce the algorithm, we let
\begin{equation}\label{eq:xi_update}
x_i^k = \arg\max_{x_i}~\langle w_i^k, x_i \rangle-f_i(x_i)
\end{equation}
for each $i\in\mc{V}$ and $k\ge 0$, which satisfies
\begin{equation}\label{eq:dual_value}
    \nabla d_i(w_i^k)=x_i^k,\quad d_i(w_i^k)=\langle w_i^k, x_i^k\rangle-f_i(x_i^k)
\end{equation}
by Danskin's Theorem \cite{bertsekas2016nonlinear}. Moreover, we choose $\mc{I}_{ij}^k=\mc{I}_{ji}^k\ni k$ as follows: For some $m>0$, if nodes $i,j$ has communicated more than $m-1$ times, it is the index set of the $m$ most recent iterations where they communicate; otherwise, it is the index set of all past iteration iterations where they communicate. A detailed implementation is provided in Algorithm \ref{alg:FDGAA}.

\begin{algorithm}
\caption{FDGM-AA}
\label{alg:FDGAA}
\begin{algorithmic}[1]
\State \textbf{Initialization:} all nodes collaboratively determine the step-size $\beta$, the Lipschitz constant $L$, the safe-guard constants $c_1,c_2$, and the initial dual iterate $w_i^0$ satisfying $\sum_{i \in \mathcal{V}} w_i^0 = 0$.
\For{$k = 0, 1, 2, \dots$}
    \For {each node $i\in\mc{V}$ satisfying $\mc{N}_i^k\ne\emptyset$}
    \State Solve $x_i^k$ according to \eqref{eq:xi_update}.
    \State Compute $d_i(w_i^k)$ and $\nabla d_i(w_i^k)$ by \eqref{eq:dual_value}.
    \For {each neighbor $j\in\mc{N}_i^k$}
        \State Send $(\nabla d_i(w_i^k), d_i(w_i^k))$ to node $j$.
        \State Receive $(\nabla d_j(w_j^k), d_j(w_j^k))$ from node $j$.
        \State Update $\mc{I}_{ij}^k$ and solve $\alpha_{ij}^k,\alpha_{ji}^k$ from \eqref{eq:coefficients}.
        \State Solve $(\bar{w}_{ij}^{k+\frac{1}{2}}, \bar{w}_{ji}^{k+\frac{1}{2}})$ from \eqref{eq:wij_bar} and \eqref{eq:tilde_w}.
        \If{safe-guard condition \eqref{eq:simple_safe} holds}
        \State Set $(w_{ij}^{k+\frac{1}{2}}, w_{ji}^{k+\frac{1}{2}})=(\bar{w}_{ij}^{k+\frac{1}{2}}, \bar{w}_{ji}^{k+\frac{1}{2}})$.
        \Else
        \State Set \((w_{ij}^{k+\frac{1}{2}}, w_{ji}^{k+\frac{1}{2}})\) according to \eqref{eq:wij1/2}.
        \EndIf
    \EndFor
    \State Compute $w_i^{k+1}$ by \eqref{eq:separable_FDGM}.
    \EndFor
\EndFor
\end{algorithmic}
\end{algorithm}

Incorporating the AA scheme into FDGM leads to a significantly higher storage cost (historical dual iterates and dual gradients), but only slightly higher communication and computation costs. Compared to FDGM, at each iteration of Algorithm \ref{alg:FDGAA}, each node communicates only one additional bit (dual value) with each neighbor, and solves one additional quadratic program \eqref{eq:coefficients} whose complexity can be much lower than solving \eqref{eq:xi_update} with complex $f_i$. Given the inherent difficulty of consensus optimization over time-varying networks, trading increased memory for accelerated convergence can be worthwhile, particularly when storage is sufficient.

\section{Convergence Analysis}\label{sec:conv_ana}
This section analyses the convergence of FDGM-AA in solving problem \eqref{eq:problem}. To this end, we impose the following conditions on the networks.

\begin{assum}[B-connectivity]\label{assum:connect}
    There exists an integer \(B > 0\) such that the union graph \( (\mathcal{V}, \bigcup_{t=k}^{k+B-1} \mathcal{E}^t)\) is connected for any \( k\ge 0\). 
\end{assum}

Assumption \ref{assum:connect} allows the network to be disconnected and is standard for time-varying networks \cite{Nedic2017,Wu2019}.

For each \(k \ge 0\), let \(\tilde{\mathcal{G}}^k = (\mathcal{V}, \tilde{\mathcal{E}}^k)\) where \(\tilde{\mathcal{E}}^k=\cup_{t=k}^{k+B-1} \mathcal{E}^t\), \(\mc{L}_{\tilde{\mathcal{G}}^k}\) be the graph Laplacian of \(\tilde{\mathcal{G}}^k\), \(\eta^k\) be the maximum degree of \(\tilde{\mathcal{G}}^k\), and  \( \tilde{\eta} = \max_{t\ge 0} \eta^{tB}\). Also define \( S_0 = \{ \mb{w} \mid D(\mb{w}) \le D(\mb{w}^0),~w_1+\ldots+w_n=\mb{0}\} \). It follows from \cite[Proposition 3]{Wu2019} that under Assumption~\ref{assum:str_convex}, $S_0$ is compact, so that $R_0=\max\{\|\mb{w}-\mb{u}\|\mid \mb{w},\mb{u}\in S_0\}<+\infty$.
\begin{thm}\label{thm:vary}
    Suppose that Assumptions \ref{assum:str_convex} and \ref{assum:connect} hold. Let the sequences \(\{\mb{w}^k\}_{k=0}^{\infty}\) and \(\{\mb{x}^k\}_{k=0}^{\infty}\) be generated by Algorithm \ref{alg:FDGAA}. If $\sum_{j \in \mathcal{N}_i^k} h_{ij}^k \le 1$ and \(\beta\) satisfies \eqref{eq:step-cond}, then
    \begin{align}
        &D(\mb{w}^{k})-D^\star\le \frac{\tau R_0^2(D(\mb{w}^0)-D^\star)}{\tau R_0^2+\underline{\lambda}\lfloor k/B \rfloor(D(\mb{w}^0)-D^\star)},\label{eq:dual_rate}\allowdisplaybreaks\\
        &\|\mb{x}^k-\mb{x}^\star\|\le \sqrt{2L(D(\mb{w}^{k})-D^\star)},\label{eq:prim_var_rate}\allowdisplaybreaks\\
        & |F(\mb{x}^k)\!-\!F^\star|\le (R_0 + \|\mb{w}^0\|)\sqrt{2L(D(\mb{w}^{k})-D^\star)},\label{eq:prim_func_rate}
\end{align}
 where $D^\star$ is the optimal value of problem \eqref{eq:fenchel}, $\bx^\star$ and $F^\star$ are the optimal solution and optimal value of problem \eqref{eq:problem}, respectively, \(\tau= \frac{3BL^2\tilde{\eta}}{\theta_2}+\frac{3}{\underline{h}\theta_1}\), and \(\underline{\lambda}= \inf_{t\in\{0,1,\cdots\}} \lambda_{\min} (\mc{L}_{\mathcal{\tilde{G}}^{kB}})\) with $\lambda_{\min}(\mc{L}_{\mathcal{\tilde{G}}^{kB}})>0$ being the smallest nonzero eigenvalue of $\mc{L}_{\mathcal{\tilde{G}}^{kB}}$.
\end{thm}
\begin{proof}
    See Appendix B.
\end{proof}

\begin{figure*}[t]
    \centering
    \includegraphics[width=0.8\textwidth]{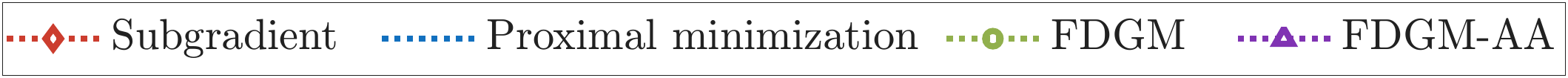}
\end{figure*}
\begin{figure*}[t]
\vspace{-0.1cm}
    \centering
    \begin{subfigure}[t]{0.32\textwidth}
        \centering
        \includegraphics[width=\linewidth,height=40mm]{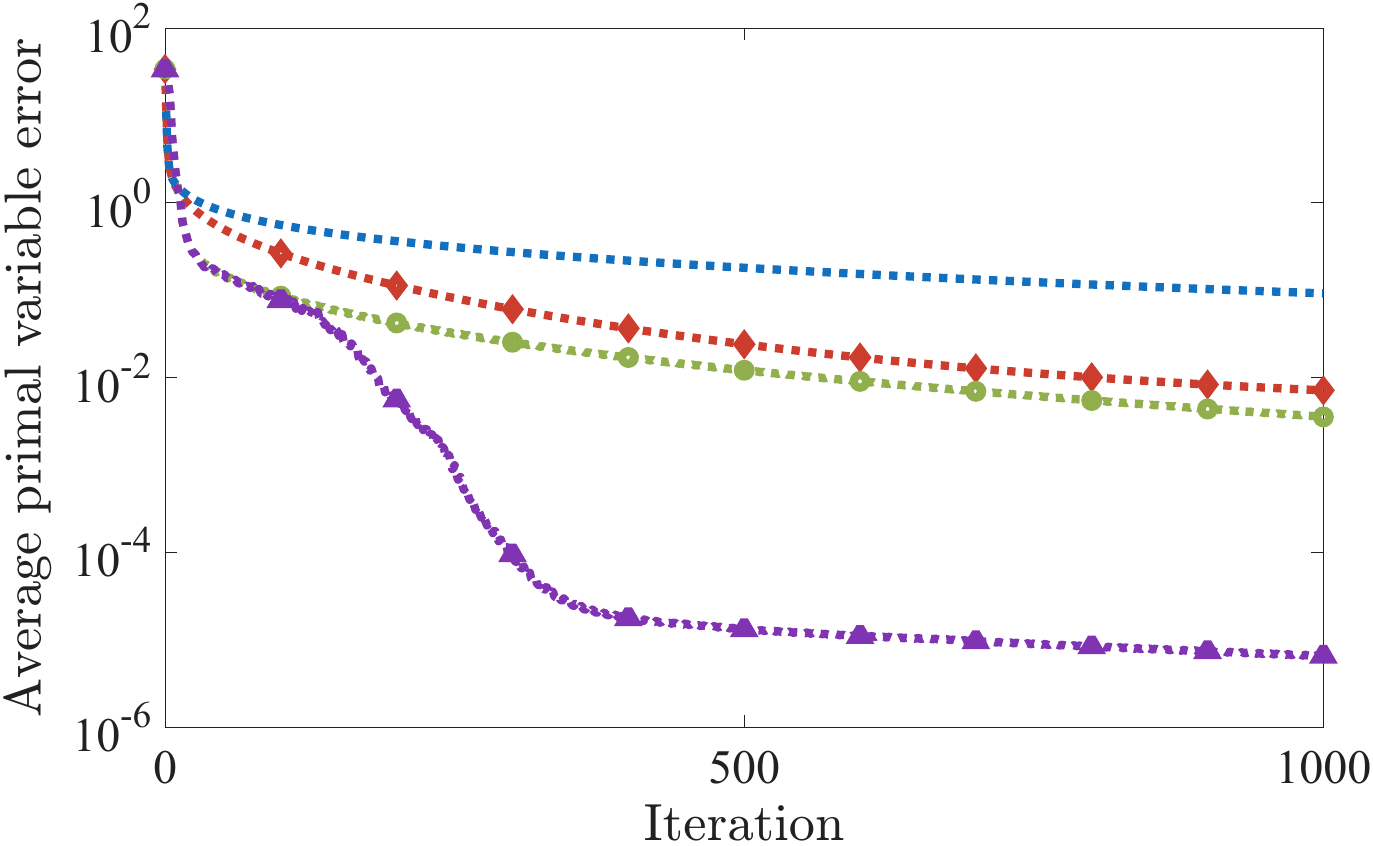}
        \caption{$B=5,\ \lambda=0.01.$}
        \label{fig:1a}
    \end{subfigure}\hfill
    \begin{subfigure}[t]{0.32\textwidth}
        \centering
        \includegraphics[width=\linewidth,height=40mm]{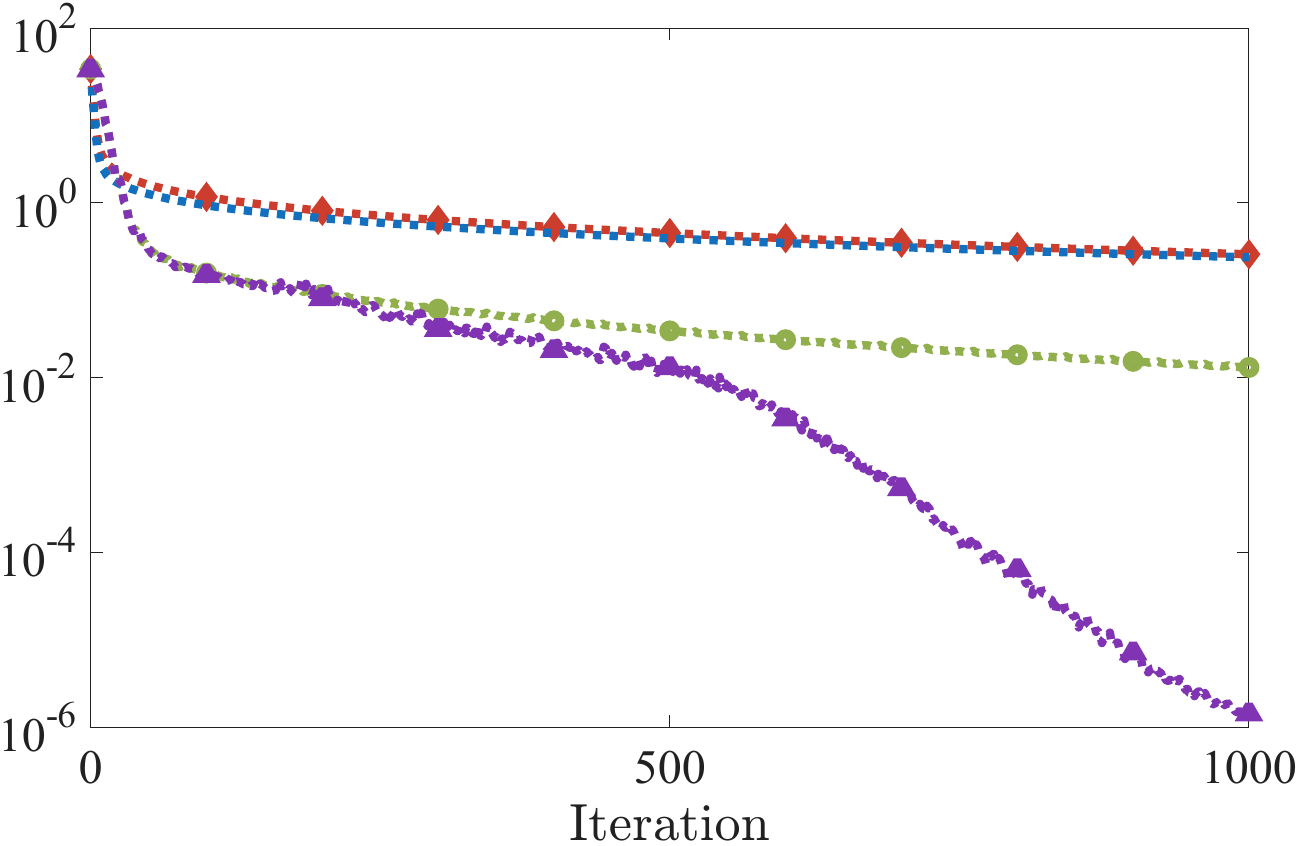}
        \caption{$B=20,\ \lambda=0.01.$}
        \label{fig:1b}
    \end{subfigure}\hfill
    \begin{subfigure}[t]{0.32\textwidth}
        \centering
        \includegraphics[width=\linewidth,height=40mm]{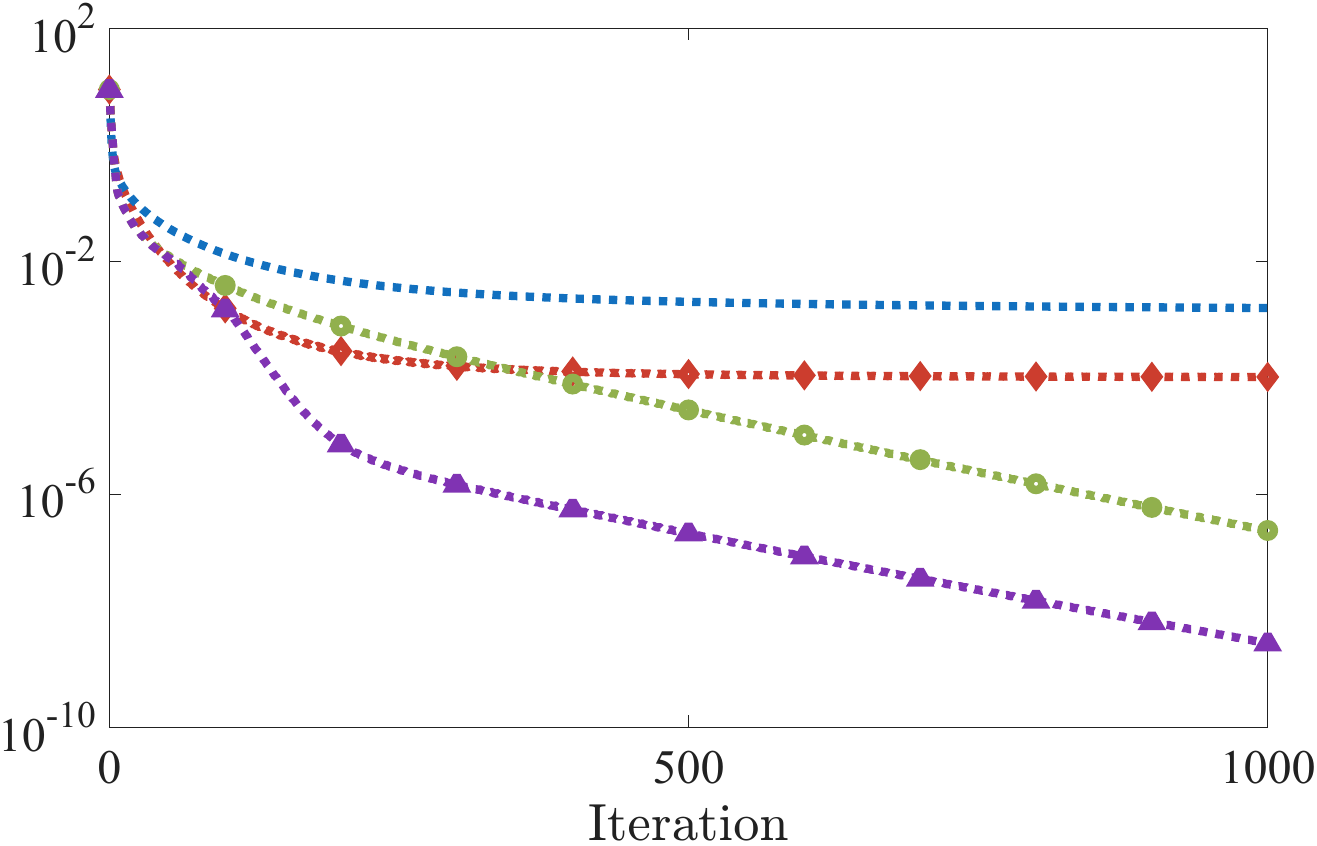}
        \caption{$B=5,\ \lambda=0.1.$}
        \label{fig:1c}
    \end{subfigure}
    \caption{Primal errors $\frac{1}{n}\sum_{i=1}^n \|x_i^k-x_i^\star\|^2$ in solving constrained logistic regression.}
    \label{fig:compare}
\end{figure*}

By Theorem \ref{thm:vary}, the dual error \(D(\mb{w}^k) - D^\star\) converges at an \(O(1/k)\) rate, and the primal errors \(\|\mb{x}^k - \mb{x}^\star\|\) and \(\|F(\mb{x}^k) - F^\star\|\) converge at an \(O(1/\sqrt{k})\) rate. 

\section{Numerical Experiment}\label{sec:Num_ex}

This section validates the effectiveness of the proposed algorithm in solving a constrained regularized logistic regression problem, where each node owns $M$ data samples, and the goal is to find a common decision that minimizes the total loss defined by data samples of all nodes:
\begin{align*}
    &\min_{x\in \mathbb{R}^d}~\sum_{i=1}^n \ell_i(x)\defeq\frac{1}{nM}\sum_{j=1}^M
  \log\left(1+e^{-b_{ij} a_{ij}^T x}\right)+\frac{\lambda}{2}\|x\|^2\\
    &\operatorname{s. t.}~x\in \bigcap_{i\in\mathcal{V}} \mc{B}(p_i,r_i),
\end{align*}
 where \(a_{ij}\in\mathbb{R}^d\) is the feature of the $j$th sample of node $i$ with the \(d\)-th entry equal to 1 and $b_{ij}$ is the corresponding label, $\lambda>0$ is the regularization parameter, and $\mc{B}(p_i,r_i)=\{x|~\|x-p_i\|\le r_i\}$ is the ball centered at $p_i\in\mathbb{R}^d$ with radius $r_i>0$. This problem corresponds to problem \eqref{eq:problem} with each \(
  f_i(x)=\ell_i(x)+\mc{I}_{\mc{B}_i(p_i,r_i)}(x)
  \) where $\mc{I}_{\mc{B}_i(p_i,r_i)}(x)$ is the indicator function of $\mc{B}_i(p_i,r_i)$. In the literature of distributed optimization over time-varying networks, only a few can handle distinct local constraints, and we compare the proposed FDGM-AA with FDGM\citep{Wu2019}, the distributed projected subgradient method\citep{Nedic2010} and the proximal minimization algorithm\citep{margellos2017distributed}. 

The experiment settings are as follows. We set $n=30$, $d=20$, and $M=20$. For each node \(i\), the numbers of the labels with \(b_{ij}=1\) and \(b_{ij}=-1\) are both \(M/2\). The first \(d-1\) entries of each \(a_{ij}\) are independently drawn from a normal distribution with mean $2$ and variance $8$, and the \(d\)-th entry is fixed to \(1\). We form a periodic time-varying network with period $B$ to satisfy Assumption \ref{assum:connect}. This is achieved by first generating a random connected graph $\mc{G}=(\mc{V},\mc{E})$ and then dividing $\mc{E}$ into $B$ subsets. At each iteration $k$, the edge set $\mc{E}^k$ is chosen as one of these subsets in a cyclic manner, so that the union of $\mc{E}^k$ over any period of $B$ equals $\mc{E}$. To maintain consistent with the references, we use fixed parameters for FDGM and FDGM-AA, and step-sizes of the form $c/k~(c>0)$ for the distributed projected subgradient method and the proximal minimization algorithm. We set the memory size $m=40$ for FDGM-AA and fine-tune the parameters of all algorithms for better performance. 

Fig. \ref{fig:compare} compares the performance of all methods and tests the effect of $B$ and $\lambda$, from which we make the following observations. First, in all settings FDGM-AA significantly and consistently outperforms other methods, which highlights the effectiveness of the proposed method. Second, the acceleration of FDGM-AA over FDGM is clear for large $k$ which corresponds to large $|\mc{I}_{ij}^k|$. Third, larger $B$ and smaller $\lambda$ leads to slower convergence of all methods, which coincides our intuition.


\section{Conclusion}\label{sec:Conc}
This paper has adapted the Anderson acceleration (AA) technique to solve distributed optimization problems over time-varying networks. In particular, we incorporated AA into the Fenchel dual gradient methods (FDGM), resulting in the FDGM-AA algorithm. To enable distributed implementation, we decomposed FDGM into a set of subproblems, employed AA to solve each subproblem, and integrated the results of all subproblems. To guarantee global convergence, we designed a novel safe-guard mechanism. Numerical experiments demonstrate the competitive performance of our method against existing ones. Future work will focus on relaxing the strong convexity assumption.

\vspace{-0.1cm}

\section*{Acknowledgments}
\vspace{-0.1cm}

This work is supported in part by the Guangdong Provincial Key Laboratory of Fully Actuated System Control Theory and Technology under grant No. 2024B1212010002, in part by the Shenzhen Science and Technology Program under grant No. JCYJ20241202125309014, and in part by the Shenzhen Science and Technology Program 
under grant No. KQTD20221101093557010.
\vspace{-0.1cm}

\section*{Appendix A: Proof of Lemma \ref{lem:always_descent}}\label{append:proof_lemma_descent}
\vspace{-0.1cm}

It is straightforward to see that if \eqref{eq:fixed_safe} holds, so does \eqref{eq:always_descent}. Therefore, to prove the lemma, it suffices to show that if \eqref{eq:fixed_safe} fails to hold and \((w_{ij}^{k+\frac{1}{2}}, w_{ji}^{k+\frac{1}{2}})\) is updated according to \eqref{eq:wij1/2}, then \eqref{eq:always_descent} still holds.

Suppose that \eqref{eq:fixed_safe} fails to hold. By the $L$-smoothness of $d_i$ and $d_j$, we have $d_i(w_{ij}^{k+\frac12})+d_j(w_{ji}^{k+\frac12}) -(d_i(w_i^k)+d_j(w_j^k))\le  \langle \nabla d_i(w_i^k),\, w_{ij}^{k+\frac12} - w_i^k\rangle+\langle \nabla d_j(w_j^k),\, w_{ji}^{k+\frac12} - w_j^k\rangle+ L\left(\big\|w_{ij}^{k+\frac12} - w_i^k \big\|^2+\big\| w_{ji}^{k+\frac12} - w_j^k \big\|^2\right)/2$, which, together with \eqref{eq:wij1/2}, yields \eqref{eq:always_descent}. 

\section*{Appendix B: Proof of Theorem \ref{thm:vary}}\label{append:proof_Thm}   
 
     The proof mainly consists of two parts: proving the non-increasing of \(\{D(\mb{w}^k)\}\), and bounding the total decrease of $D(\mb{w}^k)$ over \(B\) consecutive iterations.

    {\bf Part 1}: By \eqref{eq:separable_FDGM} and the convexity of \(d_i\), $d_i(w_i^{k+1}) \le d_i(w_i^k) + \sum_{j\in\mathcal{N}_i^k} h_{ij}^k \big(d_i(w_{ij}^{k+\frac{1}{2}}) - d_i(w_i^k)\big)$, summing which over all nodes \(i=1,\dots,n\) gives
\begin{equation}{\label{eq:Jensen}}
\begin{array}{ll}
   & D(\mb{w}^{k+1})-D(\mb{w}^k)\allowdisplaybreaks\\
   \le &\sum_{\{i,j\}\in \mathcal{E}^k} h_{ij}^k\big(d_i(w_{ij}^{k+\frac{1}{2}}) - d_i(w_i^k)\\
   & +d_j(w_{ji}^{k+\frac{1}{2}}) - d_j(w_j^k)\big)\allowdisplaybreaks\\
   \le& -\theta_1\sum_{\{i,j\}\in \mathcal{E}^k} h_{ij}^k\|\nabla d_i(w_i^k)-\nabla d_j(w_j^k)\|^2,
\end{array}
\end{equation}
where the last step uses \eqref{eq:always_descent}. By \eqref{eq:Jensen}, \(\{D(\mb{w}^k)\}\) is non-increasing.

{\bf Part 2}: We construct an upper bound on the accumulative descent of $D(\mb{w}^t)$ over $B$ steps: for some $\tau>0$,
\begin{equation}\label{eq:bounded}
    \begin{array}{ll}
       &\sum_{\{i,j\}\in\mathcal{\tilde{E}}^k}\|\nabla d_i(w_i^k)-\nabla d_j(w_j^k)\|^2\\
       \le& \tau \left(D(\mb{w}^k)-D(\mb{w}^{k+B})\right).
    \end{array}
\end{equation}
For simplicity, we define $a_i^k=\nabla d_i(w_i^k)$. To show \eqref{eq:bounded}, note that for any edge \(\{i,j\}\in \tilde{\mathcal{E}}^k\), we can find \(t_{ij}^k\in [k,k+B-1]\) such that \(\{i,j\}\in\mathcal{E}^{t^k_{ij}}\). For each \(\{i,j\}\in \tilde{\mathcal{E}}^k\), by the Cauchy-Schwarz inequality,
\begin{equation}\label{eq:grad_term}
    \begin{array}{ll}
       \|a_i^k-a_j^k\|^2
       \le& 3\|a_i^k-a_i^{t^k_{ij}}\|^2+3\|a_j^{t_{ij}^k}-a_j^k\|^2\\
       &+3\|a_i^{t^k_{ij}}-a_j^{t^k_{ij}}\|^2.
    \end{array}
\end{equation}
By \eqref{eq:Jensen},
\begin{equation}\label{eq:aitijk}
\begin{split}
     \sum_{\{i,j\}\in\mathcal{\tilde{E}}^k}\|a_i^{t^k_{ij}}-a_j^{t^k_{ij}}\|^2
     \le&\frac{1}{\underline{h}} \sum_{t=k}^{k+B-1} \sum_{\{i,j\}\in\mathcal{E}^t} h_{ij}^t \|a_i^t-a_j^t\|^2\\
     \le& \left(D(\mb{w}^k)-D(\mb{w}^{k+B})\right)/(\underline{h}\theta_1).
\end{split}
\end{equation}
By the smoothness of \(d_i\) and $t_{ij}^k-k\le B$ and using the Cauchy-Schwarz inequality, we have $\|a_i^k-a_i^{t^k_{{ij}}}\|^2 = \|\sum_{t=k}^{t^k_{{ij}}-1} (a_i^t-a_i^{t+1})\|^2
        \le B \sum_{t=k}^{k+B-1} L^2\|w_i^{t+1}-w_i^t\|^2$,
which yields
\begin{equation}\label{eq:smooth_a}
\begin{split}
   &\sum\limits_{\{i,j\}\in\mathcal{\tilde{E}}^k}\big(\|a_i^k-a_i^{t^k_{{ij}}}\|^2+\|a_j^{t_{{ij}}^k}-a_j^k\|^2\big)\\
    \le& BL^2\!\!\!\sum\limits_{\{i,j\}\in\mathcal{\tilde{E}}^k} \!\!\!\!\sum_{t=k}^{k+B-1}\!(\|w_i^{t+1}\!-\!w_i^t\|^2\!+\!\|w_j^{t+1}\!-\!w_j^t\|^2)\\
    \le& BL^2\tilde{\eta} \sum_{t=k}^{k+B-1} \sum_{i\in\mc{V}}\|w_i^{t+1}-w_i^t\|^2.
\end{split}
\end{equation}
By \eqref{eq:separable_FDGM}, the convexity of $\|\cdot\|^2$, and $\sum_{j \in \mathcal{N}_i^k} h_{ij}^k \le 1$
\begin{equation}\label{eq:wit1wit}
\begin{split}
   &\left\|w_i^{t+1}-w_i^t\right\|^2
   =\biggl\|\sum_{j\in \mc{N}_i^t} h_{ij}^t \left(w_{ij}^{t+\frac{1}{2}}-w_i^t\right)\biggr\|^2\allowdisplaybreaks\\
   &= \bigl(\sum_{j\in \mc{N}_i^t} h_{ij}^t\bigr)^2\cdot\biggl\|\frac{\sum_{j\in \mc{N}_i^t} h_{ij}^t \left(w_{ij}^{t+\frac{1}{2}}-w_i^t\right)}{\sum_{j\in \mc{N}_i^t} h_{ij}^t}\biggr\|^2\allowdisplaybreaks\\
   &\le \sum_{j\in \mc{N}_i^t}h_{ij}^t\left\|w_{ij}^{t+\frac{1}{2}}-w_i^t\right\|^2.
\end{split}
\end{equation}
By \eqref{eq:wit1wit}, \eqref{eq:always_descent}, and the first step of \eqref{eq:Jensen},
\begin{equation*}
\begin{split}
   &\sum_{i\in\mc{V}}\|w_i^{t+1}-w_i^t\|^2\le \sum_{i\in\mc{V}}\sum_{j\in\mc{N}_i^t}h_{ij}^t\|w_{ij}^{t+\frac{1}{2}}-w_i^t\|^2\\
    =& \sum\limits_{\{i,j\}\in\mathcal{E}^t}h_{ij}^t\big(\|w_{ij}^{t+\frac{1}{2}}-w_i^t\|^2+\|w_{ji}^{t+\frac{1}{2}}-w_j^t\|^2\big)\allowdisplaybreaks\\
    \le& -\frac{1}{\theta_2}\left(D(\mb{w}^{t+1})-D(\mb{w}^t)\right),
\end{split}
\end{equation*}
substituting which into \eqref{eq:smooth_a} yields
\begin{equation}\label{eq:cummulative_B}
\begin{split}
   &\sum\limits_{\{i,j\}\in\mathcal{\tilde{E}}^k}\big(\|a_i^k-a_i^{t^k_{{ij}}}\|^2+\|a_j^{t_{{ij}}^k}-a_j^k\|^2\big)\\
    \le& (BL^2\tilde{\eta}/\theta_2)\cdot(D(\mb{w}^k)-D(\mb{w}^{k+B})).
\end{split}
\end{equation}
By \eqref{eq:aitijk}, \eqref{eq:cummulative_B}, and \eqref{eq:grad_term}, we obtain \eqref{eq:bounded} with $\tau=3/(\underline{h}\theta_1)+3(BL^2\tilde{\eta}/\theta_2)$.

With the above results, we derive the convergence following the proof framework in \cite{Wu2019}. By the inequality \eqref{eq:bounded}, we have that for any $k\ge 0$,
\begin{equation}\label{eq:dual_gap}
   \begin{split}
      &(D(\mb{w}^{(k+1)B})-D^\star)-(D(\mb{w}^{kB})-D^\star)\\
      &=D(\mb{w}^{(k+1)B})-D(\mb{w}^{kB})\\
      &\le -\langle\nabla D(\mb{w}^{kB}), (\mc{L}_{\mathcal{\tilde{G}}^{kB}}\otimes I_d) \nabla D(\mb{w}^{kB})\rangle/\tau\\
      &\le -\underline{\lambda}\|P \nabla D(\mb{w}^{kB})\|^2/\tau,
   \end{split}
\end{equation}
where $P$ is the projection matrix onto $S=\{\mb{w}|~w_1+\ldots+w_n=\mb{0}\}$ and the last step uses $\operatorname{Range}(\mc{L}_{\mathcal{\tilde{G}}^{kB}}\otimes I_d)=S$ due to the connectivity of $\mathcal{\tilde{G}}^{kB}$.

Note that $\mb{w}^{kB}\in S_0$ by \eqref{eq:Jensen}, which, together with the convexity of $D$, $\mb{w}^{kB}, \mb{w}^\star\in S$, and $\mb{w}^\star\in S_0$, yields
\[
\begin{split}
    D(\mb{w}^{kB})-D^\star&\le \|P \nabla D(\mb{w}^{kB})\|\cdot\|\mb{w}^{kB}-\mb{w}^\star\|\\
                     &\le R_0\|P \nabla D(\mb{w}^{kB})\|.
\end{split}\]
Substituting the above equation into \eqref{eq:dual_gap} gives \(D(\mb{w}^{(k+1)B})-D^\star)-(D(\mb{w}^{kB})-D^\star)\le -\underline{\lambda}(D(\mb{w}^{kB})-D^\star)^2/(\tau R_0^2)\), which, according to \cite[Lemma 6, Section 2.2.1]{polyak1987introduction}, yields
\begin{equation}\label{eq:D_kB}
   \begin{split}
      D(\mb{w}^{kB})-D^\star\le \frac{D(\mb{w}^0)-D^\star}{1+\underline{\lambda}k(D(\mb{w}^0)-D^\star)/(\tau R_0^2)}.
   \end{split}
\end{equation}
By the above equation and the non-increasing property of the sequence \(\{D(\mb{w}^k)\}_{k=0}^{\infty}\), we obtain \eqref{eq:dual_rate}.

Next, we derive \eqref{eq:prim_var_rate}--\eqref{eq:prim_func_rate}. By $\mb{x}^k= \nabla D(\mb{w}^k)$, $\mb{x}^\star=\nabla D(\mb{w}^\star)$, and the $L$-smoothness of $D$, we have
\begin{equation}\label{eq:xkxstar_D}
    \begin{split}
        \|\mb{x}^k-\mb{x}^\star\|\le \sqrt{2L(D(\mb{w}^k)-D^\star)},
    \end{split}
\end{equation}
substituting which into \eqref{eq:dual_rate} gives \eqref{eq:prim_var_rate}. Since $P$ is the projection matrix of a space, we have $\|P\|_2\le 1$, which, together with $P\mb{x}^\star=\mb{0}$, yields
\begin{equation}\label{eq:Pxk_expansion}
\begin{split}
    \|P\mb{x}^k\| &= \|P(\mb{x}^k-\mb{x}^\star)\|\le \|P\|\cdot\|\mb{x}^k-\mb{x}^\star\|\\
    &\le \|\mb{x}^k-\mb{x}^\star\|\le \sqrt{2L(D(\mb{w}^k)-D^\star)}.
\end{split}
\end{equation}
Moreover, by the definition of $D(\mb{w})$,
\begin{align*}
        \langle \mb{w}^k, \mb{x}^k\rangle - F(\mb{x}^k)&\ge \langle \mb{w}^k, \mb{x}^\star\rangle - F(\mb{x}^\star)\\
        \langle \mb{w}^\star, \mb{x}^\star\rangle - F(\mb{x}^\star)&\ge \langle \mb{w}^\star, \mb{x}^k\rangle - F(\mb{x}^k).
\end{align*}
which, together with $\mb{w}^k, \mb{w}^\star\perp \bx^\star$,  leads to
\begin{equation}\label{eq:F_intm}
\begin{split}
    |F(\mb{x}^k)-F^\star|&\le \max\{|\langle \mb{w}^\star, \mb{x}^k\rangle|, |\langle \mb{w}^k, \mb{x}^k\rangle|\}\\
    &=\!\max\{|\langle \mb{w}^\star, P\mb{x}^k\rangle|, |\langle \mb{w}^k,\!P\mb{x}^k\rangle|\}\\
    &\le \max\{\| \mb{w}^k\|, \|\mb{w}^\star\|\}\cdot\|P\mb{x}^k\|,
\end{split}
\end{equation}
where the second step uses $P\mb{w}^k=\mb{w}^k$, $P\mb{w}^\star=\mb{w}^\star$, and $P=P^T$. Moreover, since $\mb{w}^k,\mb{w}^\star,\mb{w}^0\in S_0$, we have
\begin{align*}
    &\|\mb{w}^k\|\le \|\mb{w}^k-\mb{w}^0\|+\|\mb{w}^0\|\le R_0+\|\mb{w}^0\|\\
    &\|\mb{w}^\star\|\le \|\mb{w}^\star-\mb{w}^0\|+\|\mb{w}^0\|\le R_0+\|\mb{w}^0\|
\end{align*}
substituting which and \eqref{eq:Pxk_expansion} into \eqref{eq:F_intm} gives \eqref{eq:prim_func_rate}.
\bibliographystyle{plain}        
\bibliography{autosam}           
\end{document}